\documentclass[11pt]{article}
\usepackage[french,english]{babel}
\usepackage{graphicx}
\usepackage{graphics}
\usepackage{amsfonts}
\usepackage{amscd}
\usepackage{amsthm}
\usepackage{amsmath}
\usepackage{indentfirst}
\usepackage[all]{xy}
\usepackage{amssymb, amsmath, amsthm, amsgen, amstext, amsbsy, amsopn}
\usepackage{epic,eepic}
\usepackage{hyperref}
\usepackage{enumitem}
\usepackage{color}
\usepackage{dsfont}
\usepackage{comment}
\usepackage{esint}
\usepackage[margin=1.3in]{geometry}
\usepackage{mathtools}

\usepackage{appendix}

\theoremstyle{plain}
\newtheorem*{thm*}{Theorem}
\newtheorem*{cor*}{Corollary}
\newtheorem{thm}{Theorem}
\newtheorem{lemma}{Lemma}[section]
\newtheorem{prop}[lemma]{Proposition}
\newtheorem{claim}[lemma]{Claim}
\newtheorem{cor}[lemma]{Corollary}
\newtheorem{theorem}[lemma]{Theorem}
\newtheorem{question}{Question}

\theoremstyle{definition}
\newtheorem{defn}[lemma]{Definition}
\newtheorem{rem}[lemma]{Remark}

\def\supp{\text{supp}}
\def\F{\mathcal{F}}
\def\Cont{\mathrm{Cont}}
\def\Crit{\mathrm{Crit}}

\title{\textbf{$C^0$-Contact Geometry of Surfaces in 3--Manifolds}}
\author{\textbf{Baptiste Serraille, Maksim Stoki\'c}}
\date{}

\begin{document}

\maketitle

\begin{abstract}
    We prove that contact homeomorphisms preserve characteristic foliations on surfaces in contact $3$--manifolds. More precisely, since the characteristic foliation is a singular $1$–dimensional foliation, we show that singular points are mapped to singular points, and that the image of every $1$–dimensional leaf is again a $1$–dimensional leaf in the image surface. As a consequence, regular coisotropic surfaces are $C^0$–rigid. In contrast, we show that contact convexity is $C^0$–flexible by constructing a contact homeomorphism that sends a convex $2$–torus to a non-convex one.
\end{abstract}

\tableofcontents

\section{Introduction}

The \textit{characteristic foliation} $\mathcal{F}_{\Sigma}$ of a surface $\Sigma$ in a contact $3$-manifold $(Y,\xi)$ is the \textit{singular} $1$-dimensional foliation of $\Sigma$ generated by the singular line field $T\Sigma \cap \xi$. The set of critical points $\mathrm{Crit}(\mathcal{F}_{\Sigma})$ consists of those $p \in \Sigma$ where $T_p\Sigma = \xi_p$. If, in addition, we assume that:
\begin{enumerate}[label=(\roman*)]
    \item the contact structure $\xi$ is cooriented, i.e.\ $\xi = \ker \alpha$ for some contact form $\alpha$,
    \item $\Sigma$ is an oriented surface equipped with an area form $\Omega$,
\end{enumerate}
then one can define an \textit{oriented characteristic foliation} as the family of integral curves of a vector field $X$ on $\Sigma$ determined by
\[
\iota_X \Omega = \alpha|_{\Sigma}.
\]

\noindent The notion of characteristic foliation has proved to be a powerful tool in the study of contact $3$--manifolds. For instance, Giroux showed (see \cite{Gi91}) that any diffeomorphism $\phi:\Sigma_1 \to \Sigma_2$ which preserves oriented characteristic foliations, i.e.\ $\phi_*(\mathcal{F}_{\Sigma_1}) = \mathcal{F}_{\Sigma_2}$, extends to a contactomorphism between neighbourhoods, $\Phi:Op(\Sigma_1) \to Op(\Sigma_2)$. Characteristic foliations also play a central role in the classification of contact structures on $3$--manifolds. Our first result establishes that characteristic foliations are $C^0$-rigid under contact homeomorphisms.  

A \textit{contact homeomorphism} $\varphi:Y \to Y$ is a homeomorphism that arises as the $C^0$-limit of a sequence of contact diffeomorphisms $\varphi_i \in \mathrm{Cont}(Y,\xi)$. A contact analogue of the celebrated Gromov--Eliashberg theorem (see \cite{MS14}) asserts that every smooth contact homeomorphism preserves the contact structure. Our first main theorem shows the following rigidity property.

\begin{thm}\label{Thm_Foliation_Rigidity}
    Let $(Y,\xi)$ be a contact $3-$manifold, and $\Sigma \subset Y$ a smooth surface. 
    Assume $\phi:Y\to Y$ is a contact homeomorphism such that the image $\phi(\Sigma)$ is smooth. Then $\phi$ preserves the characteristic foliation set-wise 
    \[\phi_*(\F_{\Sigma}) = \F_{\phi(\Sigma)}.\]
    In particular:
    \begin{itemize}
        \item[-] Singular points are preserved: $\phi(\mathrm{Crit}(\mathcal{F}_{\Sigma})) = \mathrm{Crit}(\mathcal{F}_{\phi(\Sigma)})$.
        \item[-] For every regular leaf $\mathcal{L} \in \mathcal{F}_{\Sigma}$, the image $\phi(\mathcal{L})$ is a characteristic leaf in $\mathcal{F}_{\phi(\Sigma)}$.
\end{itemize}
\end{thm}

\begin{proof}
    The result is an immediate consequence of Propositions~\ref{Prop_Crit_to_Crit} and~\ref{Prop_Regular_Foliation_Part}.
\end{proof}

\noindent The result is local in the sense that we do not assume $\Sigma$ to be closed, nor do we assume the contact structure $\xi$ to be cooriented. If, however, $\Sigma$ is an oriented surface and $\xi = \ker \alpha$ is cooriented, then one can speak of an oriented characteristic foliation on $\Sigma$. This naturally raises the question:

\begin{question}
    Can a contact homeomorphism reverse the orientation of the characteristic foliation?
\end{question}

\noindent The question of $C^0$–coisotropic rigidity has attracted significant attention in recent years. See, for example, \cite{RZ20} and \cite{Us21} for proofs of $C^0$–coisotropic rigidity under additional assumptions on the sequence of conformal factors. Among coisotropic submanifolds, the Legendrian case—corresponding to the minimal dimension—is particularly notable, as it is the only setting where $C^0$–rigidity has been proved without imposing conditions on the conformal factors. $C^0$–Legendrian rigidity was proved in dimension three in \cite{DRS24a}, \cite{St25}, and in higher dimensions in \cite{DRS24b}.

A submanifold $C$ of a contact manifold $(Y,\xi=\mathrm{ker}\,\alpha)$ is called \textit{contact coisotropic} or \textit{regular coisotropic} if
\begin{enumerate}
    \item[(1)] for all $p\in C$ the subspace $T_pC\cap\xi_p\subset\xi_p$ is coisotropic with respect to the conformal symplectic structure on $\xi_p$. In particular,
    \[(T_pC\cap\xi_p)^{\perp_{d\alpha}}\subset T_pC\cap\xi_p,\]
    \item[(2)] the distribution $TC\cap\xi$ has constant rank.
\end{enumerate}

\noindent Some authors omit the constant-rank assumption in the definition, in which case we will call $C$ a \textit{singular coisotropic} submanifold. In contact $3$–manifolds every surface is singular coisotropic, while regular coisotropic surfaces are precisely those satisfying the condition $\mathrm{Crit}(\mathcal{F}_C)=\emptyset$. Hence the above theorem yields:

\begin{cor}
    Let $C$ be a regular coisotropic surface in a contact $3$–manifold $(Y,\xi)$. Suppose $\phi:Y \to Y$ is a contact homeomorphism such that $\phi(C)$ is smooth. Then $\phi(C)$ is a regular coisotropic surface.
\end{cor}

\noindent As we shall see, the proof of Theorem \ref{Thm_Foliation_Rigidity} consists of two main steps:
\begin{enumerate}
    \item After removing all points in the critical sets ($\mathrm{Crit}(\mathcal{F}_{\Sigma})$ and $\phi^*\mathrm{Crit}(\mathcal{F}_{\phi(\Sigma)})$), the remaining regular leaves are mapped to each other by $\phi$. The key tool is the \emph{contact hammer}, a contact analogue of Opshtein's symplectic hammer \cite{Op09}, which, together with Givental's non-displacement result for a Clifford torus in $\mathbb{RP}^3$, provides a $C^0$–characterization of regular leaves.
    
    \item Every critical point is mapped to a critical point. This requires analysing the foliation near singularities and using different techniques depending on the type: in some cases using the local orientation of leaves, and in others, applying Givental's non-displacement result.
\end{enumerate}

In \cite{Gi91}, Giroux introduced the theory of convex surfaces in contact $3$--manifolds. A surface $\Sigma$ in a contact $3$--manifold is said to be \textit{convex} if there exists a contact vector field transverse to it. For such surfaces, one can define a \textit{dividing set}, a much simpler invariant than the characteristic foliation, which plays a central role in the classification of contact structures on $3$--manifolds. Our next result shows that contact convexity is $C^0$-flexible.

\begin{thm}\label{thm.flexbility}
Let $(Y,\xi)$ be a contact $3$–manifold. There exists a convex torus $T \subset Y$ and a contact homeomorphism $\phi:Y\to Y$ such that $\phi(T)$ is a smooth, non-convex torus.
\end{thm}

\begin{proof}
    The proof follows from Theorem \ref{Thm_Convexity_Flexibility}.
\end{proof}

\noindent Our strategy is to show that, in the $C^0$–limit, a property of convex surfaces may fail. For a smooth convex surface, every closed leaf of the characteristic foliation is \textit{non-degenerate}: the associated \textit{Poincaré return map}, which measures how nearby leaves wind around a periodic orbit, is non-degenerate. The proof idea is to construct a $C^0$–example where this non-degeneracy condition breaks down. Although our construction is carried out on a torus, the same argument applies in a neighbourhood of any closed leaf on an arbitrary surface.

\subsection*{Acknowledgments}

\noindent The authors warmly thank Lev Buhovsky for invaluable discussions regarding this project. The authors also thank Georgios Dimitroglou-Rizell, Du\v{s}an Joksimovi\'c, Álvaro del Pino Gómez, and Sobhan Seyfaddini for their interest in this project. M.S. conducted part of this work during a visit to Zürich, supported by the FIM Institute, and is supported by Uppsala University. B.S. was supported partially by ERC Starting Grant 851701.

\section{Characteristic foliations on surfaces}

In this section, we recall the definition and give some properties of characteristic foliations of hypersurfaces inside 3-dimensional contact manifolds. In Section \ref{subsec.sets_that_can_arise}, we characterize the sets that can arise as singular sets of a characteristic foliation. In Section \ref{subsec.local_structure_near_sing}, we give results concerning the local structure near singularities.

\medskip

\noindent Let $(Y^3,\xi=\ker\alpha)$ be a co-oriented contact $3$--manifold, and let $\Sigma \subset Y$ be an oriented embedded surface. The \textit{characteristic foliation} $\F_\Sigma$ on $\Sigma$ is the (possibly singular) line field defined by
\[
T_p\F_\Sigma := \xi_p \cap T_p\Sigma
\qquad\text{for all } p \in \Sigma.
\]
This defines a smooth singular foliation on $\Sigma$ whose \textit{singular set} is
\[
\Crit(\F_\Sigma)
:= \{\, p \in \Sigma\mid \xi_p = T_p\Sigma \,\}.
\]

\noindent Let $\omega_\Sigma$ be the orientation $2$-form on $\Sigma$. Then $\F_\Sigma$ is oriented by the unique vector field $X$ (defined up to positive scalar) satisfying:
\[
\iota_X \omega_\Sigma =\alpha|_{\Sigma}.
\]
Equivalently, $X$ is the unique line field tangent to $\ker\alpha$ such that $d\alpha(X,n)>0$ for any positively oriented pair $(X,n)$ in $T_p\Sigma$ with $\alpha(n)>0$.

\subsection{Sets that can arise as the critical set of a characteristic foliation}\label{subsec.sets_that_can_arise}

In this section, we show that critical sets of characteristic foliations are included in 1-dimensional smooth submanifolds. Moreover, for each closed sets of the closed interval, we construct a characteristic foliation which has such set as critical set. This last result does not appear in the literature as far as we know.

\begin{prop}\label{Prop_Crit_in_1D_mfld}
Let $(Y, \xi)$ be a contact manifold and $\Sigma \subset Y$ be a smooth hypersurface, then the critical set of the characteristic foliation, $\Crit (\mathcal F_\Sigma)$, is closed and contained in a disjoint union of smooth 1-dimensional submanifolds of $\Sigma$.
\end{prop}

\begin{proof}
Let $p$ be a critical point of the characteristic foliation of $\Sigma$ and consider an open subset $U$ of $p$ in $\Sigma$ and identify it with a portion of the $xy$-plane of $\mathbb R^3$. Then $\beta =\alpha\lvert_\Sigma$ can be written as $\beta_x dx+ \beta_y dy$ with $\beta_x$ and $\beta_y$ smooth functions. The critical set is the intersection set of the zero set of $\beta_x$ and $\beta_y$, it is then closed from the continuity of $\beta_x$ and $\beta_y$.

Since $\beta_p=0$ and $\alpha$ is non-degenerate, we have that $d\beta_p =(\partial_x \beta_y - \partial_y \beta_x) dx \wedge dy \neq 0$. This means that $\partial_x \beta_y \neq\partial_y \beta_x$, in particular they cannot be both 0, e.g. $\partial_y \beta_x\neq 0$. The function $\beta_x \colon U \to \mathbb R$ is a smooth submersion at 0 locally around $p$ hence its zero set is a 1-dimensional smooth manifold.
\end{proof}

\noindent The converse is also true.

\begin{prop}
Let $F \subset I=(a,b)$ with $a,b \in \mathbb R$ be a compact subset of $\mathbb R$. Then there exist a contact manifold $(Y, \xi)$, a smooth hypersurface $\Sigma \subset Y$ as well as a smooth embedding $\iota\colon I \hookrightarrow \Sigma$ such that
\[\Crit(\mathcal F_\Sigma)= \iota(F).\]
\end{prop}

\begin{proof}
Let $g \colon \mathbb R \to \mathbb R$ be a function such that its zero set is exactly $F$ and $g'(x) <1$ for all $x \in \mathbb R$. The first condition can be built using Whitney theorem for the second one can always assume that $g$ is constant at infinity (from the compactness of $F$) and then replace $g$ by $cg$ for some real number $ c \in \mathbb R^*$ to obtain small enough derivatives.

Then let us consider the 1-form $\alpha = y dx + g(x) dy+ dz$ on $\mathbb R^3$, we have that $d \alpha=dy\wedge  dx + g'(x) dx \wedge dy + 0$. Thus,
\[\alpha \wedge d\alpha=(g'(x)-1)dx \wedge dy \wedge dz \neq 0.\]
This implies, that $\alpha$ is a contact form. We consider now $\Sigma$ as the $xy$-plane, and write $\beta=\alpha\lvert_\Sigma=y dx +g(x) dy$ and $\iota \colon \mathbb R \to \Sigma$ as the natural inclusion of the line $\{y=0\}$ in $\Sigma$. Then 
\[\Crit(\mathcal F_\Sigma)=\{y=0\} \cap \{g(x)=0\}=\iota(F).\]
\end{proof}

\subsection{Local structure near singularities in characteristic foliations}\label{subsec.local_structure_near_sing}

In this section, we study the characteristic foliations near the singularities. In particular, we prove Lemma \ref{lem.two_leaves_lemma} that tells us that each singular point of a characteristic foliation sees at least 2 leaves converging to it. This result will be key in order to prove Proposition \ref{Prop_Crit_to_Crit}. We then give a local model of the characteristic foliation near 1-dimensional sets of singularities.

\begin{lemma}[Two leaves Lemma]\label{lem.two_leaves_lemma}
Let $\Sigma \subset (Y^3, \ker \alpha)$ be an embedded surface, and let 
$p \in \Crit(\mathcal{F}_{\Sigma})$ be a singular point of the characteristic foliation on $\Sigma$.  
Then there exist at least two distinct characteristic leaves of $\Sigma$ that either both converge to $p$ as $t \to +\infty$ or both converge to $p$ as $t \to -\infty$.
\end{lemma}

\begin{proof}
Choose local coordinates $(x,y) \in \mathbb{R}^2$ near $p$ so that $p$ corresponds to the origin and is singular.  
Let $X$ generate the characteristic foliation.  
Since $\mathrm{div}\,X(0,0) \neq 0$, we have $\mathrm{rank}\,DX(0,0) \geq 1$.

If $\mathrm{rank}\,DX(0,0) = 1$, the claim follows from Lemma~\ref{Lemma:Rank_1_Case}.  
In the hyperbolic case $\mathrm{rank}\,DX(0,0) = 2$, the Hartman--Grobman Theorem gives a topological conjugacy with its linearization, whose eigenvalues $\lambda_1,\lambda_2$ have $\mathrm{Re}\,\lambda_i \neq 0$. Depending on their signs, the origin is a \emph{sink} ($\mathrm{Re}\,\lambda_1, \mathrm{Re}\,\lambda_2 < 0$), a \emph{source} ($\mathrm{Re}\,\lambda_1, \mathrm{Re}\,\lambda_2 > 0$), or a \emph{saddle} ($\mathrm{Re}\,\lambda_1 < 0 < \mathrm{Re}\,\lambda_2$). In all cases, there exist at least two distinct trajectories converging to $(0,0)$ as $t \to +\infty$ or $t \to -\infty$.
\end{proof}

\begin{lemma}[Local model near one-dimensional singularity]\label{Lemma_Local_1dim_Crit}
Let $\Sigma \subset (Y^3, \ker \alpha)$ be an embedded surface, and let $T \subset \Crit(\mathcal{F}_\Sigma)$ be a $1$-dimensional component of its singular set. For every $p \in T$, there exists a neighbourhood $\mathcal{U} \subset Y$ of $p$ such that the pair $(\mathcal{U}, \mathcal{U} \cap \Sigma)$ is contactomorphic to $(\mathbb{R}^3, \{z=0\})$ endowed with the contact form $\alpha_+ = dz + ydx$ or $\alpha_- = dz - ydx$.
\end{lemma}

\begin{proof}
Choose local coordinates $(x,y)$ on $\Sigma$ near $p \in T$ in which the singular set $T$ is the $y$-axis $\{x=0\}$.  
Let $X$ be the vector field generating the characteristic foliation of $\Sigma$ in these coordinates.  
Since $\mathrm{div}\,X(0,y) \neq 0$ (which follows from the non-degeneracy of the contact form), we conclude that $DX(0,y)$ has exactly one non-zero eigenvalue.  
Let $\lambda(y)$ denote this unique non-zero eigenvalue of $DX(0,y)$, and let $V(y)$ be a normalized eigenvector corresponding to $\lambda(y)$.\\

\noindent By Lemma~\ref{Lemma:Rank_1_Case}, for every point $(0,y)$ on the $y$-axis there exists a smooth embedding
\[
    L_y : (-\varepsilon, \varepsilon) \hookrightarrow \mathbb{R}^2
\]
satisfying:
\begin{itemize}
    \item $L_y$ is transverse to the $y$-axis and $\left.\frac{d}{dt}\right|_{t=0} L_y(t) = V(y)$;
    \item $L_y^+ := L_y|_{t>0}$ lies in a characteristic leaf contained in $\{y>0\}$;
    \item $L_y^- := L_y|_{t<0}$ lies in a characteristic leaf contained in $\{y<0\}$;
    \item if $\lambda(y) > 0$, the characteristic leaves converge to the $y$-axis in the outward direction; if $\lambda(y) < 0$, they converge inward.
\end{itemize}

noindent Since $y \mapsto V(y)$ is smooth, it follows that the map
\begin{equation}\label{Equation:(x,y)-L_y}
    (x,y) \mapsto L_y(x)
\end{equation}
is a diffeomorphism between neighbourhoods of the $y$-axis, fixing the $y$-axis pointwise and mapping horizontal lines $\{y=\mathrm{const}\}$ to $\mathrm{Im}\,L_y$.\\

The characteristic foliation of $\{z=0\} \subset (\mathbb{R}^3, \ker \alpha_{\pm})$ is generated by $\pm x \partial_x$.  
Therefore, the diffeomorphism \eqref{Equation:(x,y)-L_y} preserves oriented characteristic foliations, and hence, by a theorem of Giroux~\cite{Gi91}, there exists a contactomorphism between the corresponding tubular neighbourhoods.
\end{proof}

\begin{lemma}\label{Lemma:Rank_1_Case}
Let $X$ be a vector field on $\mathbb{R}^2$ with $X(0,0)=0$ and $\mathrm{rank}\,DX(0,0)=1$, whose eigenvalues are $0$ and $\lambda\neq 0$. Then there exist two distinct integral curves $\gamma_1,\gamma_2:\mathbb{R}\to\mathbb{R}^2$ such that:
\begin{enumerate}[label=(\roman*)]
    \item If $\lambda<0$, then $\lim_{t\to+\infty}\gamma_i(t)=(0,0)$; if $\lambda>0$, then $\lim_{t\to-\infty}\gamma_i(t)=(0,0)$.
    \item The restriction $\left.\frac{X}{\|X\|}\right|_{\mathrm{Im}\,\gamma_i}$ extends smoothly to $(0,0)$ and lies in the eigenspace of $\lambda$.
\end{enumerate}
\end{lemma}

\begin{proof}
\noindent Assume $\lambda > 0$ (the $\lambda < 0$ case is analogous). After a linear change of coordinates:
\[ X(x,y) = (\lambda x + u(x,y), v(x,y)) \]
where $u,v = \mathcal{O}(\|(x,y)\|^2)$. For $0 < \varepsilon < a$, define:
\[ \mathcal{D}_{\varepsilon,a} = \{\varepsilon \leq x \leq a, |y| \leq |x|\} \subset \mathcal{D}_{a} = \{0<x\leq a, |y|\leq |x|\}. \]

\noindent For small enough $a$ in $\mathcal{D}_{a}$, the Taylor's expansion implies:
\begin{align*}
    \dot{x} &= \lambda x + \mathcal{O}(x^2)\\
    \dot{y} &= \mathcal{O}(x^2).
\end{align*}
\noindent Thus all trajectories in $\mathcal{D}_a$ flow strictly in positive $x$-direction ($\dot{x} > 0$), and the slope satisfies
\begin{equation}\label{Eq:slope}
    \lvert\mathrm{slope}(X(x,y))\rvert =\Big\lvert\frac{v(x,y)}{\lambda x+v(x,y)}\Big\rvert= \frac{\mathcal{O}(x)}{\lambda + \mathcal{O}(x)} < 1,
\end{equation}
\noindent so trajectories cannot exit through $|y| = |x|$. Hence any trajectory entering $\mathcal{D}_{\varepsilon,a}$ must exit through $\{a\} \times [-a,a]$. By continuity and compactness, there exists $\gamma_1$ in $\mathcal{D}_a$ with $\lim_{t\to -\infty}\gamma_1(t) = (0,0)$. Similarly, we find $\gamma_2$ in $\{x < 0\}$ with $\lim_{t\to -\infty}\gamma_2(t) = (0,0)$, proving (i).

For (ii), \eqref{Eq:slope} implies $\lim_{t\to-\infty}\mathrm{slope}(\gamma_1(t))=0$, so $X/\|X\|$ extends continuously to $(0,0)$ with limit $(1,0)$, an eigenvector of $DX(0,0)$ for $\lambda$. The existence of 
\[ \lim_{t\to-\infty}\frac{\partial^n}{(\partial t)^n}\left(\mathrm{slope}(\gamma_1(t))\right) \]
for all $n\geq 1$ (for example via induction and L'Hôpital's rule) shows the extension of the slope function of $X$ from $\mathrm{Im}\,\gamma_1$ to $(0,0)$ is smooth, and hence the extension of $X/\|X\|$ from $\mathrm{Im}\,\gamma_1$ to $(0,0)$ is smooth because 
\[\frac{\gamma'_1(t)}{\|\gamma_1'(t)\|}=\left(\frac{1}{\sqrt{1+\mathrm{slope}(\gamma_1(t))^2}},\frac{\mathrm{slope}(\gamma_1(t))}{\sqrt{1+\mathrm{slope}(\gamma_1(t))^2}}\right).\]
\noindent The proof for $\gamma_2$ follows similar arguments.
\end{proof}

\section{Isotopy extension theorems}

Let $S_0$ and $S_1$ be two surfaces in a contact $3$--manifold $(Y, \xi = \ker \alpha)$. Suppose there exists a diffeomorphism $\phi: S_0 \to S_1$ such that $\phi((S_0)_\xi) = (S_1)_\xi$, where $(S_i)_\xi$ denotes the oriented characteristic foliation induced by $\xi$ on $S_i$. Then, by a theorem of Giroux \cite{Gi91}, the surfaces $S_0$ and $S_1$ admit contactomorphic neighbourhoods. In what follows, we will require a \emph{one-parametric} and a \emph{relative} version of this result, which we state and prove in this section.

\begin{theorem}\label{TheoremIsotopyExtension}
    Let $V\subset S$ be an open subset of a closed surface $S$ in a $3$--dimensional contact manifold $(Y,\xi)$, and let $j_t:S\rightarrow(Y,\xi=\mathrm{ker}\,\alpha),\,t\in[0,1],$ be an isotopy of embeddings such that for all $t\in[0,1]$ we have
    \begin{itemize}
        \item $j_t|_{V}=j_0|_{V}$,
        \item $j_t$ and $j_0$ induce the same characteristic foliation on $S$.
    \end{itemize}
    \noindent Then, there exists a contact isotopy $\psi_t:Y\rightarrow Y$ which satisfies
    \begin{enumerate}
        \item $\psi_t\circ j_0=j_t$,
        \item $\psi_t$ is supported inside $Op(j_t(S\setminus V))$.
    \end{enumerate}
    Extend $j_t$ to a family of embeddings $\phi_t:S\times\mathbb{R}\rightarrow M$ satisfying $\phi_t|_{V\times\mathbb{R}}=\phi_0|_{V\times\mathbb{R}}$. Then we can ensure that $\psi_t$ is generated by a contact Hamiltonian function $H_t$ compactly supported inside $\phi_t((S\setminus V)\times\mathbb{R})$.
\end{theorem}

\begin{proof}
    The assumptions say that all $\phi_t^*\alpha$ induce the same characteristic foliation on $S\times\{0\}$.
    \begin{claim}
        For $\varepsilon>0$ small enough, there exists a smooth family of embeddings $\varphi_t:S\times(-\varepsilon,\varepsilon)\hookrightarrow S\times\mathbb{R}$ such that $\varphi_0$ is the standard inclusion, $\varphi_t|_{S\times\{0\}}=\varphi_0|_{S\times\{0\}}$, $\varphi_t|_{V\times(-\varepsilon,\varepsilon)}=\varphi_0|_{V\times(-\varepsilon,\varepsilon)}$ and
        \begin{equation}\label{MoserTrick}
            \varphi_t^*\phi_t^*\alpha=\lambda_t\phi_0^*\alpha,\text{ for some }\lambda_t:S\times(-\varepsilon,\varepsilon)\rightarrow(0,+\infty).
        \end{equation}
    \end{claim}
    \begin{proof}
        The proof is essentially a version of Moser's trick, similar to the proof of Gray stability theorem. Let $\alpha_t=\phi_t^*\alpha$. By differentiating equation (\ref{MoserTrick}) and with the help of Cartan's formula, the condition a vector field $X_t$ has to satisfy so that its flow will pull back $\phi_t^*\alpha$ to $\lambda_t\phi_0^*\alpha$ is
        \begin{equation}\label{MoserTrickExpanded}
        \dot{\alpha}_t+d(\alpha_t(X_t))+i_{X_t}d\alpha_t=\mu_t\alpha_t,
        \end{equation}
        where $\mu_t=\frac{d}{dt}(\log\lambda_t)\circ\varphi_t^{-1}$.\\

        \noindent Write $X_t=h_tR_t+Y_t$ with $R_t$ the Reeb vector field of $\alpha_t$ and $Y_t\in\mathrm{ker}\,\alpha_t$. The equation (\ref{MoserTrickExpanded}) translates into
        \begin{equation}\label{MoserTrickExpanded2}
            \dot{\alpha}_t+dh_t+i_{Y_t}d\alpha_t=\mu_t\alpha_t.
        \end{equation}

        \noindent Our aim is to find a solution $X_t=h_tR_t+Y_t$ such that $h_t|_{S\times\{0\}}=0$, $Y_t|_{S\times\{0\}}=0$ as well as $h_t|_{V\times(-\varepsilon,\varepsilon)}=0$ and $Y_t|_{S\times(-\varepsilon,\varepsilon)}=0$.\\
        
        \noindent From the assumption that all $\alpha_t=\phi_t^*\alpha$ induce the same characteristic foliation on $S\times\{0\}$ we get that there exists a function $f_t:S\times\{0\}\rightarrow(0,+\infty)$ which satisfies
        \begin{equation}\label{Condition_foliation}
            \forall\,\xi\in T(S\times\{0\})\quad\alpha_t(\xi)=f_t\cdot\alpha_0(\xi).
        \end{equation}
        \noindent We impose the following condition on the function $h_t$:
        \begin{equation}\label{Conditon_ht}
            h_t|_{S\times\{0\}}=0\text{ and }\dot{\alpha}_t+dh_t=\mu_t\alpha_t\text{ along }S\times\{0\}.
        \end{equation}
        \noindent Using (\ref{Condition_foliation}) we see that the two conditions in (\ref{Conditon_ht}) imply that $\mu_t|_{S\times\{0\}}=(df_t/dt)/f_t$. Moreover, if we indeed set $\mu_t|_{S\times\{0\}}=(df_t/dt)/f_t$, they can be simultaneously satisfied.\\
        
        \noindent The condition $\phi_t|_{V\times\mathbb{R}}=\phi_0|_{V\times\mathbb{R}}$ implies that $\alpha_t|_{V\times\mathbb{R}}=\alpha_0|_{V\times\mathbb{R}}$, meaning that $\dot{\alpha}_t|_{V\times\mathbb{R}}=0$, so we can impose another condition on $h_t$ which is compatible with (\ref{Conditon_ht})
        \begin{equation}\label{Condition_ht2}
            h_t|_{V\times\mathbb{R}}=0.
        \end{equation}
        
        \noindent Lastly, pick a function $h_t$ which satisfies (\ref{Conditon_ht}) and (\ref{Condition_ht2}), then we get $\mu_t$ by inserting a Reeb vector field $R_{t}$ in (\ref{MoserTrickExpanded2}). Since $d\alpha_t|_{\mathrm{ker}\,\alpha_t}$ is non-degenerate, the equation (\ref{MoserTrickExpanded2}) admits a unique solution $Y_t\in\mathrm{ker}\,\alpha_t$. From (\ref{Conditon_ht}),(\ref{MoserTrickExpanded2}) we conclude that $Y_t|_{S\times\{0\}}=0$, and from (\ref{Condition_ht2}),(\ref{MoserTrickExpanded2}) we conclude that $Y_t|_{V\times\mathbb{R}}=0$.\\

        \noindent The flow of a vector field $X_t=h_tR_t+Y_t$ defines an isotopy $\varphi_t$ that fixes $S\times\{0\}$ point-wise, and thus it is defined for all $t\in[0,1]$ in a neighbourhood of $S\times\{0\}$.
    \end{proof}
    \noindent The above claim implies that $\phi_t\circ\varphi_t:S\times(-\varepsilon,\varepsilon)\hookrightarrow (Y,\xi)$ is a smooth family of contact embeddings which additionally satisfies $\phi_t\circ\varphi_t(q,0)=j_t(q)$ and $\frac{d}{dt}(\phi_t\circ\varphi_t)|_{V\times(-\varepsilon,\varepsilon)}=0$. A contact Hamiltonian function $\widetilde{H}_t\in C^{\infty}(\phi_t\circ\varphi_t(S\times(-\varepsilon,\varepsilon)))$ generating $\phi_t\circ\varphi_t$ is defined as $\widetilde{H}_t=\alpha(\frac{d}{dt}(\phi_t\circ\varphi_t))$. Let $\chi_t:C^{\infty}_c(Y)$ be a cut-off function which satisfies $\chi|_{j_t(S\setminus\overline{V})}=1$ and $\chi|_{Y\setminus Op(j_t(S\setminus\overline{V}))}=0$. Finally, define a function $H_t\in C^{\infty}_c(Y)$ as $H_t(x)=\chi_t(x)\cdot\widetilde{H}_t(x)$ for $x\in\phi_t\circ\varphi_t(S\times(-\varepsilon,\varepsilon))$ and $H_t(x)=0$ otherwise. Now we get $\psi_t$ as a contact flow generated by a contact Hamiltonian function $H_t$.
\end{proof}

\begin{defn}[Support on the Time Interval]
    Let $\gamma_{t}:X\rightarrow Y,\,t\in I$ be a smooth isotopy. The support of $\gamma_{t}$ at time $t_0\in I$ is defined as
    \[\mathrm{supp}\,\gamma_{t_0}:=\mathrm{closure}\{\gamma_{t_0}(x)\mid\frac{d}{dt}\Big\vert_{t=t_0}\gamma_{t}(x)\neq 0\}\subset Y.\]
    \noindent The support of the isotopy $\gamma_t$ on the interval $J\subset I$ is 
    $\mathrm{supp}_{t\in J}\gamma_{t}:=\bigcup_{t\in J}\mathrm{supp}\,\gamma_{t}$.
\end{defn}

\begin{defn}[$\varepsilon$-Piecewise Supported Isotopy]\label{epsilonPSI}
    Let $\phi_t:S\hookrightarrow(Y,\xi),\,t\in[0,1]$ be a smooth family of embeddings of a surface $S$ to a $3$--dimensional contact manifold $(Y,\xi)$. We say that $\phi_t$ is an \textit{$\varepsilon$-piecewise supported isotopy} (abbreviated $\varepsilon$-PSI) if one can find two finite families $\{U^0_i\}_{i=1}^{n_0}$,$\{U^1_i\}_{i=1}^{n_1}$ of disjoint open subsets of $Y$ such that
    \begin{itemize}
        \item $(\forall\,j\in\{0,1\})(\forall\,1\leq i\leq n_j)$ the diameter of the set $U^j_i$ is less than $\varepsilon$,
        \item  $\mathrm{supp}_{t\in[0,1/2]}\,\phi_t\subset\bigsqcup_{i=1}^{n_0}U^0_i$
        \item $\mathrm{supp}_{t\in[1/2,1]}\,\phi_t\subset\bigsqcup_{i=1}^{n_1}U^1_i$.
    \end{itemize}
\end{defn}

\begin{prop}\label{PropositionPSI}
    Let $\phi_t:S\rightarrow (Y,\xi),\,t\in[0,1]$ be an $\varepsilon$-PSI of a surface $S$ in a $3$--dimensional contact manifold $(Y,\xi)$. Assume that all $\phi_t$ induce the same characteristic foliation on $S$. Then there exists a contact isotopy $\Psi_t:Y\rightarrow Y$ such that for each $t\in[0,1]$ we have
    \begin{enumerate}
        \item $\Phi_t\circ\phi_0=\phi_t$,
        \item $d_{C^0}(\Phi_t,\mathrm{Id})<2\varepsilon$,
        \item $\mathrm{supp}\,\Phi_t\subset Op(\mathrm{supp}\,\phi_t)$.
    \end{enumerate}
\end{prop}

\begin{proof}
    Since $\phi_t$ is an $\varepsilon$-PSI, we can find two families of disjoint open sets of diameters bounded from above by $\varepsilon$, $\{U^0_i\}_{i=1}^{n_0}$ and $\{U^1_i\}_{i=1}^{n_1}$, such that $\mathrm{supp}_{t\in[0,1/2]}\phi_t\subset\bigsqcup_{i=1}^{n_0}U^0_i$ and $\mathrm{supp}_{t\in[1/2,1]}\phi_t\subset\bigsqcup_{i=1}^{n_1}U^1_i$. Consider two open sets $V_0,V_1\subset S$ defined as
    \begin{equation}
        V_0=S\setminus(\phi_0)^{-1}(\mathrm{supp}_{t\in[0,1/2]}\phi_t),\quad V_1=S\setminus(\phi_{1/2})^{-1}(\mathrm{supp}_{t\in[1/2,1]}\phi_t).
    \end{equation}
    \noindent We now get the result by applying Theorem \ref{TheoremIsotopyExtension} two times, first for $\phi_t,\,t\in[0,1/2]$ and $V_0$, and then for $\phi_t,\,t\in[1/2,1]$ and $V_1$. Indeed, we will get a contact isotopy $\Phi_t:Y\rightarrow Y,\,t\in[0,1]$, which extends $\phi_t$ and moreover satisfies
    \begin{equation}\label{randomEq}
    \begin{split}
        (\forall\,t\in[0,1/2])\quad\mathrm{supp}\,\Phi_t\subset Op(\mathrm{supp}_{t\in[0,1/2]}\phi_t)\subset\bigsqcup_{i=1}^{n_0} U^0_i,\\
        (\forall\,t\in[1/2,1])\quad\mathrm{supp}\,\Phi_t\subset Op(\mathrm{supp}_{t\in[1/2,1]}\phi_t)\subset\bigsqcup_{i=1}^{n_1} U^1_i.
    \end{split}
    \end{equation}
    \noindent From (\ref{randomEq}) we conclude that the trajectory $t\mapsto\Phi_t(p)$ of any point $p\in Y$ (that is not fixed by the whole isotopy) visits at most one open set in the family $\{U^0_i\}_{i=1}^{n_0}$ and at most one set in the family $\{U^1_i\}_{i=1}^{n_1}$. Since the diameter of any of those visited sets is less than $\varepsilon$, we get $d_{C^0}(\Phi^t,\mathrm{Id})<2\varepsilon$.
\end{proof}

\section{Contact hammers}

In this section, we define the notion of contact hammer. This definition is inspired by Opshtein's symplectic hammers. We prove that contact hammers allow us to characterize the leaves on regular sets of a characteristic foliation.\\

\noindent Let $(Y,\xi)$ be a contact $3$–manifold and $\Sigma\subset Y$ a surface.  
An open set $\mathcal U\subset Y$ is a \emph{Darboux chart adapted to $\Sigma$} if  
\[
(\mathcal U,\Sigma\cap\mathcal U,\xi)\cong
\bigl(\mathbb R^{3},\mathbb R^{2}_{yz},\ker(dz-x\,dy)\bigr).
\]
\noindent Such a Darboux chart exists in a neighbourhood of every point $p\in\Sigma$ where the characteristic foliation is regular. The complement $\mathcal U\setminus\Sigma$
has two open components $\mathcal U^{\pm}$. 

\begin{defn}[$\varepsilon$–contact hammer]
Fix $p,q\in\Sigma\cap\mathcal U$ and $\varepsilon>0$, and let
$B_{\varepsilon}(p),B_{\varepsilon}(q)\subset\mathcal U$ be the open
$\varepsilon$–balls.
An \emph{$\varepsilon$–contact hammer} from $p$ to $q$ is an isotopy
$\{\Phi_{t}\}_{t\in[0,1]}\subset\overline{\Cont}_{0,c}(\mathcal U,\xi)$
such that there exists two open balls $B_p \subset B_\varepsilon(p)$ and $B_q \subset B_\varepsilon(q)$ with the following properties
\begin{enumerate}[label=(\roman*)]
    \item $\Phi_{0}=\mathrm{Id}$;
    \item for every $t>0$, $\Phi_{t}(\Sigma\cap B_p)\subset\mathcal{U^{+}}$,
    \item for every $t>0$, $\Phi_{t}(\Sigma\cap B_q)\subset\mathcal{U^{-}}$,
    \item if $\zeta\in\Sigma\setminus\bigl(B_p\cup B_q\bigr)$ then $\Phi_{t}(\zeta)\in\Sigma\setminus\bigl(B_p\cup B_q\bigr)$ for all $t\in[0,1]$.
\end{enumerate}
\end{defn}

\begin{rem}
    We denote by $\overline{\Cont}_{0,c}(\mathcal{U},\xi)$ the group of compactly supported homeomorphisms that are $C^0$-limits of contactomorphisms in $\Cont_{0,c}(\mathcal{U},\xi)$. If $\phi \colon  (\mathcal U,\xi) \to (\mathcal V, \xi')$ is a contact homeomorphism from the contact manifolds $\mathcal U$ and $\mathcal V=\phi(\mathcal U)$, then 
    \[\overline{\Cont}_{0,c}(\mathcal V, \xi')=\phi \,\overline{\Cont}_{0,c}(\mathcal U, \xi) \,\phi^{-1} \coloneqq \{\phi \psi \phi^{-1} \mid \psi \in \overline{\Cont}_{0,c}(\mathcal U, \xi)\}.\]
    Indeed, let us show that the last set is included in the first one. By a symmetry argument this is enough to finish the proof. Let $\psi_k \to \psi$ be a sequence of contactomorphisms uniformly converging to $\psi$ and $\phi_k \colon \mathcal U_k \subset \mathcal U \to \phi(\mathcal U_k) \subset \mathcal V$ a sequence of contactomorphisms that converges to $\phi$ for the uniform topology. Provided that the support of $\psi_k$ is compactly included in $\mathcal U_k$, we can define the contactomorphism $\phi_k \psi_k \phi_k^{-1}$ on $\mathcal V$ by extending it to identity where it is still not defined. Since $\phi_k \psi_k \phi_k^{-1}$ converges to $\psi \phi \psi^{-1}$ we get our answer.
\end{rem}

\begin{figure}[h]
    \centering
    \includegraphics[scale=0.8]{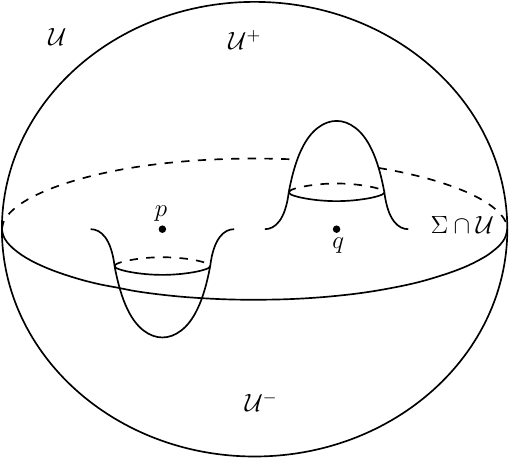}
    \caption{Contact hammer}
    \label{fig:contact_hammer}
\end{figure}

\begin{lemma}\label{Lemma_Hammer_on_single_leaf}
Let $\Sigma$ be a surface in a contact $3$–manifold $(Y,\xi)$ and let $\mathcal U\subset Y$ be a Darboux chart adapted to $\Sigma$. If $p,q\in\Sigma\cap\mathcal U$ lie on the same characteristic leaf $\mathcal L\subset\mathcal F_{\Sigma}$, then for every $\varepsilon>0$ there exists an $\varepsilon$–contact hammer between $p$ and $q$.
\end{lemma}

\begin{proof}
We construct an $\varepsilon$-contact hammer in the adopted chart $(\mathbb{R}^3,\ker dz-xdy)$ between any two points $(y_1,0)$, $(y_2,0)$ where $y_1<y_2$. The contact vector field generated by a Hamiltonian function $H$ is given by
\[X_H = \left(\frac{\partial H}{\partial y}+x \dfrac{\partial H}{\partial z}\right)\partial_x-\frac{\partial H}{\partial x}\partial_y + \left(H - x\frac{\partial H}{\partial x}\right)\partial_z.\]

\noindent For $0<\delta<\tfrac{1}{2}(y_{2}-y_{1})$ let $f_{\delta}\in C^{\infty}(\mathbb{R};[0,\infty))$ with $\supp\, f_{\delta}\subset[y_{1}-\delta,y_{2}+\delta]$ and
\[
\begin{cases}
f'_{\delta}(t)>0 &\text{for }y_{1}-\delta<t<y_{1}+\delta,\\
f'_{\delta}(t)=0 &\text{for }y_{1}+\delta\le t\le y_{2}-\delta,\\
f'_{\delta}(t)<0 &\text{for }y_{2}-\delta<t<y_{2}+\delta.
\end{cases}
\]

\noindent Let $g_{\delta}\in C^{\infty}(\mathbb{R};[0,\infty))$ such that
\[
\supp\,g_{\delta}\subset[-\delta,\delta],\qquad
g_{\delta}(t)>0\ \text{for }|t|<\delta,\qquad
g'_{\delta}(t)=0\ \text{for }|t|\leq\delta/2.
\]

\noindent Define $H(x,y,z):=f_{\delta}(y)\,g_{\delta}(x)\,g_{\delta}(z)$, and let $Q_i=(y_i-\delta, y_i+\delta)\times(-\delta,\delta),\,i\in\{1,2\}$. Then the $x$-coordinate evolves as $\dot{x}=f'_\delta(y)g_\delta(x)g_\delta(z)$, thus we have
\[
\dot{x}(0,y,z)=
\begin{cases}
0      & (y,z)\notin Q_1\cup Q_2,\\
>0     & (y,z)\in Q_1,\\
<0     & (y,z)\in Q_2.
\end{cases}
\]
\noindent Hence the flow keeps $\{x=0\}\setminus(Q_1\cup Q_2)$ invariant, pushes points in $Q_1$ to $x>0$, and pushes those in $Q_2$ to $x<0$. When $\delta>0$ is small enough we can pull back this flow to $\Sigma$ and get an $\varepsilon$-contact hammer between $p$ and $q$.
\end{proof}

\begin{prop}\label{Prop_Contact_Hammer}
    Let $\Sigma$ be a surface inside a contact 3-manifold $(Y,\xi)$, and let $\mathcal{U}\subset Y$ be an open Darboux ball disjoint from $\Crit(\F_{\Sigma})$. If for every $\varepsilon>0$ the points $x,y\in \Sigma\cap\mathcal{U}$ admit an $\varepsilon$-contact hammer, then $\mathcal{L}^{\Sigma}_x=\mathcal{L}^{\Sigma}_y$.
\end{prop}

\begin{proof}
    We first describe the setting of Givental's theorem on the non-displaceability of certain Legendrians in $\mathbb{RP}^3$. Let  
    \[S^3 \coloneqq \{(z,w) \in \mathbb{C}^2 \mid |z|^2 + |w|^2 = 1\}\]  
    be the standard 3-sphere endowed with its standard contact structure  
    \[\xi_{\mathrm{std}} = \ker(\alpha_{\mathrm{std}}) = \ker\Big(\sum_i x_i\, dy_i - y_i\, dx_i\Big)\big|_{S^3}.\]
    The Hopf $S^1$-action on $S^3$ is given by $\theta \cdot (z,w) = (e^{i\theta} z, e^{i\theta} w)$ and defines the fibration  
    \[S^1 \to S^3 \xrightarrow{\pi} S^2 = \mathbb{CP}^1.\]  
    This action preserves both the contact structure and the standard contact form $\alpha_{\mathrm{std}}$. In particular, quotienting by the subgroup $\mathbb{Z}/2 \subset S^1$ yields the fibration  
    \[S^1 \to \mathbb{RP}^3 \xrightarrow{\hat{\pi}} S^2.\]
    Under this quotient, $\mathbb{RP}^3$ inherits a contact form, still denoted $\alpha_{\mathrm{std}}$, as well as the contact structure, again denoted $\xi_{\mathrm{std}}$. When $S^2$ is equipped with the standard symplectic form $\omega_0$ of area $4\pi$, we have $d\alpha_{\mathrm{std}} = \hat{\pi}^* \omega_0$. The Clifford torus  
    \[T \coloneqq \Big\{(z,w) \mid |z| = |w| = \frac{1}{\sqrt{2}}\Big\} = \Big\{\Big(\frac{e^{i\theta_1}}{\sqrt{2}}, \frac{e^{i\theta_2}}{\sqrt{2}}\Big) \mid (\theta_1, \theta_2) \in S^1 \times S^1\Big\} \subset S^3\]  
    projects to a torus, still denoted $T$, inside $\mathbb{RP}^3$. Moreover, it projects under $\hat{\pi}$ onto an equator $\gamma \subset S^2$ (i.e., $\gamma$ divides $S^2$ into two regions of equal area). One computes $\alpha_{\mathrm{std}}|_{T} = (d\theta_1 + d\theta_2)/4$, so that the characteristic foliation $\mathcal{F}_T$ consists of parallel Legendrian circles, each being a lift of $\gamma$. The Reeb vector field of $\alpha_{\mathrm{std}}$ is tangent to $T$ and maps each lift to another; in other words, $T$ is Reeb-invariant.  

    Givental's result \cite{Gv90} asserts that no element of $\Cont_0(\mathbb{RP}^3, \xi_{\mathrm{std}})$ displaces any lift $\widetilde{\gamma}$ of $\gamma$ from $T$. Consequently, for all $\psi \in \Cont_0(\mathbb{RP}^3)$,  \[\hat{\pi}(\psi(\widetilde{\gamma})) \cap \gamma \neq \emptyset.\]  
    This statement extends immediately to $\overline{\Cont}_0(\mathbb{RP}^3)$.\\
    
    \begin{figure}[h]
    \centering
    \includegraphics[width=0.9\textwidth]{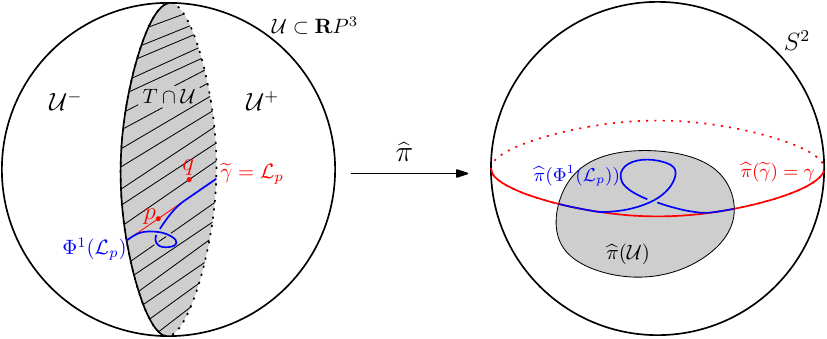}
    \caption{An $\varepsilon$-contact hammer applied to a leaf of the Chekanov torus, with projection to the sphere $S^2$.}
    \label{fig:Legendrian_displacement}
    \end{figure}
    
    Let $\mathcal{U} \subset \mathbb{RP}^3$ be a Darboux chart adapted to $T$. Assume, for contradiction, that there exist $p,q \in T$ such that for all $\varepsilon > 0$, there exists an $\varepsilon$-contact hammer between points $p$ and $q$ lying on different leaves of $\mathcal{F}_T$. For $\varepsilon$ small enough, let  
    \[\Phi^t \in \overline{\Cont}_{0,c}(\mathcal{U}) \subset \overline{\Cont}_0(\mathbb{RP}^3)\]  
    be an $\varepsilon$-contact hammer between $p$ and $q$ with $\mathcal{L}_p \neq \mathcal{L}_q$ and  
    \[B_\varepsilon(p) \cap \mathcal{L}_q = B_\varepsilon(q) \cap \mathcal{L}_p = \emptyset.\]  
    Pick a small Reeb-invariant neighbourhood $N_p$ of the Reeb orbit of $p$ such that $\Phi^1(T)$ avoids $N_p$. Then $V \coloneqq \hat{\pi}(N_p) \subset S^2$ avoids $\hat{\pi}(\Phi^1(\mathcal{L}_p))$. Moreover, since the contact hammer maps the open set $B_p$ (see the definition of a contact hammer) around $p$ to $\mathcal{U}^+$, the projection $\hat{\pi}(\Phi^1(\mathcal{L}_p))$ lies entirely on one side of $\gamma$.\\

    \noindent Pick $H \colon S^2 \to \mathbb{R}$ such that its restriction to $\gamma$ decreases outside of $V$. Then, for small $\tau > 0$, $\varphi_H^\tau(\hat{\pi}(\mathcal{L}_p))$ is disjoint from $\gamma$. In particular, $\widetilde{\varphi}_H^\tau \circ \Phi^1$ displaces $\mathcal{L}_p$ from $T$, where $\widetilde{\varphi}_H^\tau$ denotes the lift of $\varphi_H^\tau$. Since  
    \[\widetilde{\varphi}_H^\tau \circ \Phi^1 \in \overline{\Cont}_0(\mathbb{RP}^3),\]  
    this contradicts Givental's non-displacement result.
\end{proof}

\section{$C^0$-Rigidity of Characteristic Foliations}

In this section we prove the $C^0$-rigidity of characteristic foliations. As announced in the introduction we will prove first that the leaves are preserved away from the critical points. In a second time, we will then show that the singular sets are also preserved. In the first part, we will use the contact hammers while in the second part, we will utilize the results from Section \ref{subsec.sets_that_can_arise} and Section \ref{subsec.local_structure_near_sing}.

\subsection{Proof of Theorem \ref{Thm_Foliation_Rigidity}}

\begin{prop}\label{Prop_Regular_Foliation_Part}
    Let $\phi \in \overline{\Cont}(Y,\xi)$ be a contact homeomorphism with $\phi(\Sigma)$ smooth, and let $\F_\Sigma$, $\F_{\phi(\Sigma)}$ be oriented singular foliations. Define:
    \[
    S = \Sigma \setminus \bigl(\Crit\,\F_\Sigma \cup \phi^{-1}(\Crit\,\F_{\phi(\Sigma)})\bigr).
    \]
    Then the leafs of characteristic foliation on $S$ are set-wise preserved $\phi_*( \F_S) = \F_{\phi(S)}$.

\end{prop}

\begin{proof}
    It suffices to show that for every $p \in S$, the leaves $\phi(\mathcal{L}^{S}_p)$ and $\mathcal{L}^{\phi(S)}_{\phi(p)}$ coincide near $\phi(p)$. Let $U = Op(p) \subset Y$ be a Darboux neighborhood of $p$ such that $\phi(U)$ is a Darboux neighborhood of $\phi(p)$. Suppose, for contradiction, that there exists a point $q \in \mathcal{L}^{S}_p \cap U$ such that $\phi(q) \notin \mathcal{L}^{\phi(S)}_{\phi(p)}$. For $\varepsilon > 0$ small enough, let $\Psi_t$ be an $\varepsilon$-contact hammer between $p$ and $q$ (see Lemma \ref{Lemma_Hammer_on_single_leaf}). Then, the pushforward  
    \[\Phi_t := \phi \circ \Psi_t \circ \phi^{-1} \]
    \noindent defines an $\varepsilon'$-contact hammer between $\phi(p)$ and $\phi(q)$, where $\varepsilon' > 0$ satisfies $\varepsilon' \to 0$ as $\varepsilon \to 0$. This contradicts Proposition \ref{Prop_Contact_Hammer}, completing the proof.
\end{proof}

\begin{prop}\label{Prop_Crit_to_Crit}
    Let $\Sigma$ be a closed oriented surface in a contact 3-manifold $(Y,\xi=\ker\alpha)$, and let $\phi \in \overline{\Cont}(Y,\xi)$ be a contact homeomorphism such that $\phi(\Sigma)$ is smooth. Then $\phi$ preserves the critical set of the characteristic foliation:
    \[
    \phi(\Crit(\F_{\Sigma})) = \Crit(\F_{\phi(\Sigma)}).
    \]
\end{prop}

\noindent We now use the next lemma in order to deduce from Proposition \ref{Prop_Crit_to_Crit} a statement of the action of $\phi$ on the orientation of the leaves of the characteristic foliations of $\Sigma$ and $\phi(\Sigma)$.

\begin{lemma}\label{lemma_orientation_on_cc}
Let $\Sigma$ be a surface endowed with any singular foliation $\mathcal F$ such that $\Crit (\mathcal F)$ is a closed subset of $\Sigma$. Let assume that a leaf $\mathcal L$ of the foliation is oriented, then one can orient the leaves of the connected component of $\Sigma \setminus \Crit(\mathcal F)$ that contains $\mathcal L$ in a unique way if orientable.
\end{lemma}

\begin{proof}
We want to orient a leaf $\mathcal L'$ that lies in the same connected component as $\mathcal L$ . Let $x \in \mathcal L'$ and $y \in \mathcal L$, then we can find a smooth path $\gamma \colon [0,1] \to \Sigma \setminus \Crit(\mathcal F)$ that joins $y$ to $x$. We can cover the path by a finite number of charts of the foliation. On each chart the orientation is uniquely defined by one leaf. This means that the orientation of the leaf that contains $y$ is uniquely determined by the orientation of the leaf that contains $x$. This finishes the proof.
\end{proof}

Since a contact homeomorphism between two smooth hypersurface preserves the foliation near points that are regular in both surfaces and since those foliations are naturally oriented, we can formulate the following corollary.

\begin{cor}\label{Cor_Oriented_Char_F_to_Oriented_Char_F}
    Let $Y$ be a contact manifold and $\Sigma \subset Y$ a hypersurface. Let $\phi$ be a contact homeomorphism such that the image $\phi(\Sigma)$ is smooth. Then the homeomorphism $\phi \lvert_{\Sigma} \colon \Sigma \to \phi(\Sigma)$ either completely preserves or completely reverses the orientation of the characteristic foliation on each connected component of
    \[\Sigma \setminus \left(\Crit(\mathcal F_\Sigma) \cup \phi^{-1}(\Crit(\mathcal F_{\phi(\Sigma)})\right).\]
\end{cor}

\noindent We split the proof of Proposition \ref{Prop_Crit_to_Crit} in two lemmas that we prove in the next section.

\begin{lemma}\label{Lemma_0D-Crit_to_0D-Crit}
With the same hypothesis as in Proposition \ref{Prop_Crit_to_Crit}, for each point $p \in \Crit(\mathcal F_\Sigma)$ that is not in a 1-dimensional set of critical points, we have $\phi(p) \in \Crit \big( \mathcal F _{\phi(\Sigma)} \big)$.
\end{lemma}

\begin{lemma}\label{Lemma_1D-Crit_to_1D-Crit}
With the same hypothesis as in Proposition \ref{Prop_Crit_to_Crit}, for each point $p \in \Crit(\mathcal F_\Sigma)$ that lies inside a 1-dimensional set of critical points, we have $\phi(p) \in \Crit\big(\mathcal{F}_{\phi(\Sigma)}\big)$.
\end{lemma}

\noindent The proposition \ref{Prop_Crit_to_Crit} follows now without much more work.

\begin{proof}[Proof of Proposition \ref{Prop_Crit_to_Crit}]
From Lemma \ref{Lemma_0D-Crit_to_0D-Crit} and \ref{Lemma_1D-Crit_to_1D-Crit} we get that for all point $p \in \Crit(\mathcal F_\Sigma)$, $\phi(p)$ is a critical point of $\phi(\Sigma)$. Thus 
\[\phi\big(\Crit (\mathcal F _\Sigma) \big)\subset \Crit \big(\mathcal F_{\phi(\Sigma)}\big).\] 
Using the same lemmas but on $\phi^{-1}$ prove the other inclusion and thus concludes the proof of the proposition.
\end{proof}

\subsection{Proof of Lemmas \ref{Lemma_0D-Crit_to_0D-Crit} and \ref{Lemma_1D-Crit_to_1D-Crit}}

In this section we prove the lemmas \ref{Lemma_0D-Crit_to_0D-Crit} and \ref{Lemma_1D-Crit_to_1D-Crit} that, in the setting of Proposition \ref{Prop_Crit_to_Crit}, deals with the image of a critical points under the respective conditions that it does not belongs to a 1-dimensional set of critical points or that it does not. We start off with Lemma \ref{Lemma_0D-Crit_to_0D-Crit}, the proof relies on the fact that in this case, one can look at the orientation of the characteristic leaves to infer a contradiction.

\begin{proof}[Proof of Lemma \ref{Lemma_0D-Crit_to_0D-Crit}]
We work under the assumptions and with the notations of Proposition \ref{Prop_Crit_to_Crit}. Let $p$ be a critical point of the characteristic foliation that is not contained in a 1-dimensional component of the set of critical points and let us assume by contradiction that $\phi(p) \not \in \Crit(\mathcal F_{\phi(\Sigma)})$. Since the set $\Crit(\mathcal F_{\phi(\Sigma)})$ is closed and does not contain $\phi(p)$, we can find an open ball $\mathcal U$ around $p$ such that 
\[\mathcal U \cap \phi^{-1}(\Crit(\mathcal F_{\phi(\Sigma)}))= \emptyset.\] 
The set $\Crit(\mathcal F _\Sigma)$ is a strict closed subset of $\ell$, a 1-dimensional smooth manifold, according to Proposition \ref{Prop_Crit_in_1D_mfld}. The set $\mathcal U \setminus \ell$ has exactly two connected component and these two sets are connected in $\mathcal U \setminus \Crit(\mathcal F _\Sigma)$ since $\Crit(\mathcal{F}_\Sigma)\cap \mathcal U$ is a strict subset of $\ell$. This implies by Corollary \ref{Cor_Oriented_Char_F_to_Oriented_Char_F}, that the homeomorphism $\phi\lvert_{\mathcal U \setminus \Crit(\mathcal F _\Sigma)}$ sends the oriented foliation $\mathcal F _\Sigma  \lvert_{\mathcal U\setminus \Crit(\mathcal F_\Sigma)}$ to the oriented foliation $\mathcal F _{\phi(\Sigma)}$. However on $\Sigma$, the point $p$ has two leaves converging to it whose orientation both points toward $p$ or outward $p$ while on $\phi(\Sigma)$ (Lemma \ref{Lemma:Rank_1_Case}), the point $\phi(p)$ has \textit{exactly} two leaves converging to it one towards and one outwards (Lemma \ref{Lemma_Local_1dim_Crit}). This is the contradiction we wanted.
\end{proof}

We prove now Lemma \ref{Lemma_1D-Crit_to_1D-Crit}, the proof relies on the fact that if 1-dimensional sets of singularities are not sent onto 1-dimensional set of singularities then in the image there is a contactomorphism that preserves $\phi(\Sigma)$ but displaces an open set of $\phi(\Crit(\mathcal F_\Sigma))$. We use this contactomorphism and the local modal near singularities to Reeb-displace $\Crit(\mathcal F_\Sigma)$ from itself on $\Sigma$. We show that this violates again Givental's result.

\begin{proof}[Proof of Lemma \ref{Lemma_1D-Crit_to_1D-Crit}]
We work under the assumptions and with the notations of Proposition \ref{Prop_Crit_to_Crit}. Let $p$ be a critical point of the characteristic foliation that is contained in a 1-dimensional component of the set of critical points and let us assume by contradiction that 
\[\phi(p) \not \in \Crit(\mathcal F_{\phi(\Sigma)}).\]
From the closeness of $\Crit(\mathcal F_{\phi(\Sigma)})$ and the standard neighbourhood near 1-dimensional critical sets (Lemma \ref{Lemma_Local_1dim_Crit}), one can find an open Darboux ball $\mathcal U \subset M$ around $p$ such that 
\[\mathcal U \cap \phi^{-1}(\Crit(\mathcal F_{\phi(\Sigma)}))= \emptyset\]
and $(\mathcal U, \mathcal U \cap \Sigma)$ is contactomorphic to a neighbourhood of 0 of $(\mathbb R^3, xy\text{-plane})$ endowed with the standard contact structure $\ker(dz-ydx)$.

\begin{figure}[h]
    \centering
    \includegraphics[width=0.7\textwidth]{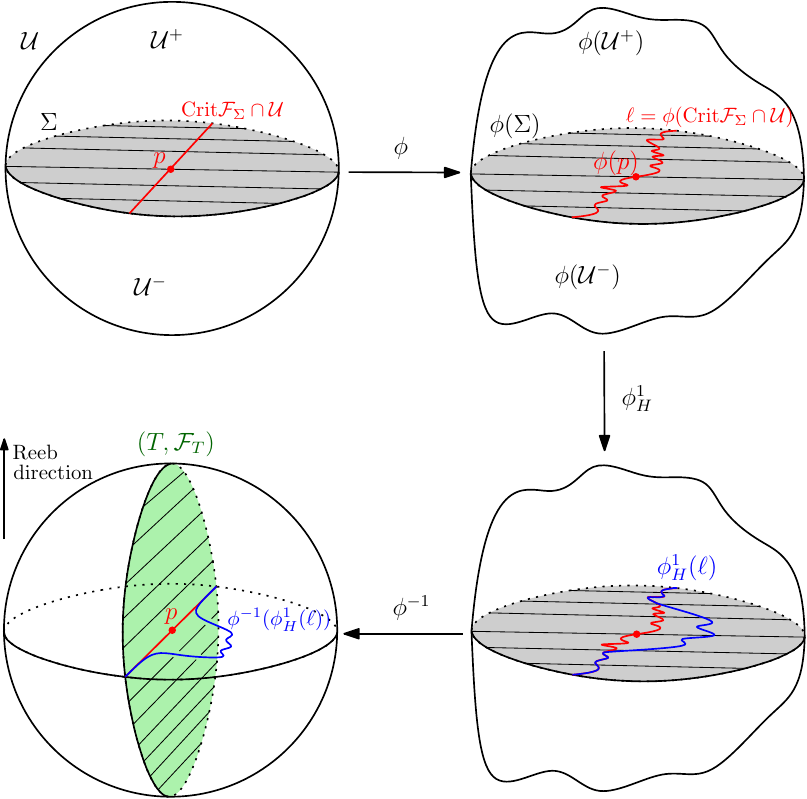}
    \caption{On the left, the Darboux ball around $p$ and on the right the image of this Darboux ball by $\phi$. At the bottom left the picture is viewed as in $\mathbb R \mathbb P^3$ and $T$ represents the Chekanov torus.}
    \label{fig:1D_Singularity}
\end{figure}

Since $\phi(p)$ is a regular point of the foliation $\mathcal F_{\phi(\Sigma)}$, we can identify $\phi(\mathcal U \cap \Sigma)$ with a plane foliated by parallel lines while $\ell \coloneqq \phi(\Crit(\mathcal F_\Sigma)\cap\,\mathcal U)$ is some topological submanifold, see Figure \ref{fig:1D_Singularity}. Moreover, from Proposition \ref{Prop_Regular_Foliation_Part}, each leaf of the characteristic foliation of $\phi(\Sigma \cap \mathcal U)$ corresponds under $\phi^{-1}$ to two leaves of the characteristic foliation of $\mathcal{U}$ glued along a point of $\Crit(\mathcal F_{\Sigma})\cap \mathcal U$. This means that each leaf in $\phi(\mathcal U)$ contains a unique point of $\ell$. In the rest of the proof we will use the fact that it is possible to displace a point from $\ell$, such that the image of $\ell$ is to the right, but not necessarily strictly (see Figure \ref{fig:1D_Singularity}). We will then use this fact to displace $\Crit(\mathcal F_\Sigma)$ from the Reeb direction, thus contradicting Givental's non-diplaceability result.

The plane $\Sigma$ splits the ball $\mathcal{U}$ in two parts that we denote $\mathcal U^-$ and $\mathcal U^+$. Then $\phi(\mathcal{U}\cap \Sigma)$ splits the ball $\phi(\mathcal{U})$ into two parts that we denote $\phi(\mathcal{U}^-)$ and $\phi(\mathcal{U}^+)$. Let $H \colon Y \to \mathbb R$ be a smooth Hamiltonian, compactly supported in $\phi(\mathcal U)$, such that $H\lvert_{\phi(\Sigma)}\equiv 0$, $H\lvert_{\phi(\mathcal U^-)} \leq 0$, $H\lvert_{\phi(\mathcal U ^+)} \geq 0$ and $dH_{\phi(p)}\neq 0$. Then, from the definition of the characteristic foliation the flow $\varphi_H^t$ follows the leaves of the foliation and every point goes in the same direction (not necessarily strictly). Without loss of generality, we may assume that every point goes to the right. This implies that there exists some open ball $\mathcal V \ni p$ such that $\varphi^1_H(\ell) \cap \mathcal V = \emptyset$. Then the isotopy $\psi^t= \phi^{-1} \varphi_H^t \phi \in \overline{\Cont}_{0,c}(\mathcal U)$ of contact homeomorphisms verifies 
\[\psi^1(\Crit(\mathcal F_{\Sigma})) \cap \phi^{-1}(\mathcal V) =\emptyset.\]
From the local model near 1-dimensional singularity, we can identify the 1-jet neighbourhood of $\Crit(\mathcal{F}_\Sigma)$ around $p$ with the 1-jet neighbourhood of $\widetilde{\gamma}$ (using the notation of the proof of Proposition \ref{Prop_Contact_Hammer}), the Clifford torus then becomes the trace of $\Crit(\mathcal{F}_\Sigma)$ under the Reeb flow. This torus is transverse to $\Sigma$ since the Reeb vector field is transverse to $\Sigma$ near $p$. We then run the same argument as in Proposition \ref{Prop_Contact_Hammer} in order to obtain a contradiction.
\end{proof}

\section{$C^0$-Flexibility of Convex Surfaces}

In this section, we prove Theorem \ref{thm.flexbility} by proving Theorem \ref{Thm_Convexity_Flexibility}. In Section \ref{subsec.Setup}, we describe the set-up. In particular we say which convex torus we send to a non-convex one. In Section \ref{subsec.pf_of_flexibility}, we prove Theorem \ref{Thm_Convexity_Flexibility} assuming the technical Proposition \ref{MainProposition} proven in the last section.

\subsection{Setup and Basic Properties}\label{subsec.Setup}

\noindent Let $\mathbb{T}(x,y) = (\mathbb{R}/\mathbb{Z})^2$ denote the flat torus, and consider the contact manifold 
\[
Y := \mathbb{T}(x,y) \times \mathbb{R}(r), \qquad \xi = \ker(dy - r\,dx).
\]
\noindent For any smooth function $f \colon \mathbb{S}^1 \to \mathbb{R}$, define an embedded surface
\[
S_f := \{(x, y, f(y)) \mid (x, y) \in \mathbb{T}(x,y)\} \subset Y.
\]

\noindent We identify $\mathbb{S}^1 = \mathbb{R}/\mathbb{Z}$ with the interval $[-\tfrac{1}{2}, \tfrac{1}{2})$ (with endpoints identified), and represent $\mathbb{T}(x,y)$ as the unit square $[-\tfrac{1}{2}, \tfrac{1}{2})^2$ with standard identifications.

\noindent Let $f_{\infty} \in C^\infty(\mathbb{S}^1)$ be a function satisfying:
\begin{itemize}
  \item $f_{\infty}(y) = -y^3$ for $y \in [-\tfrac{1}{4}, \tfrac{1}{4}]$,
  \item $f_{\infty}(-\tfrac{1}{2}) = 0$ and $f_{\infty}'(-\tfrac{1}{2}) \neq 0$,
  \item $f_{\infty}(y) > 0$ for $y \in (-\tfrac{1}{2}, 0)$ and $f_{\infty}(y) < 0$ for $y \in (0, \tfrac{1}{2})$.
\end{itemize}

\noindent Fix $0 < \varepsilon < \tfrac{1}{4}$. Define the subspace
\[
\mathcal{C}_{\varepsilon}(f_{\infty}) \subset C^{\infty}(\mathbb{S}^1)
\]
\noindent to consist of all functions $f$ satisfying:
\begin{enumerate}[label=(C\arabic*)]
  \item $f(y) = f_{\infty}(y)$ for $y \notin (-\varepsilon, \varepsilon)$, and $\max_{y \in [-\varepsilon, \varepsilon]} |f(y)| < \varepsilon$,
  \item $f(y) = -f(-y) < 0$ for all $y \in (0, \varepsilon)$,
  \item There exists $0 < \delta < \varepsilon$ such that $f(y) = -y$ for all $y \in (-\delta, \delta)$.
\end{enumerate}

\noindent Similarly, define a subspace
\[
\widetilde{\mathcal{C}}_{\varepsilon}(f_{\infty}) \subset C^{\infty}(\mathbb{S}^1)
\]
consisting of functions $f$ satisfying:
\begin{enumerate}
  \item[(C1)] and (C2) as above,
  \item[($\widetilde{\text{C3}}$)] There exists $0 < \delta < \varepsilon$ such that $f(y) = -y^3$ for all $y \in (-\delta, \delta)$.
\end{enumerate}

\begin{claim}
The surface $S_{f_{\infty}} \subset (Y, \xi = \ker \alpha)$ is not convex, whereas for every $0 < \varepsilon < \tfrac{1}{4}$ and every $f \in \mathcal{C}_{\varepsilon}(f_{\infty})$, the surface $S_f \subset (Y, \xi)$ is convex.
\end{claim}

\begin{proof}
Let $i \colon \mathbb{T} \hookrightarrow Y$, $i(x,y) = (x,y,f_{\infty}(y))$, be the parametrization of $S_{f_{\infty}}$, and write $\beta := i^*\alpha = dy - f_{\infty}(y)\,dx$. Assume for contradiction that $S_{f_{\infty}}$ is convex. Then there exists a function $u \in C^{\infty}(\mathbb{T})$ such that
\[
ud\beta + \beta \wedge du 
= \left(u\,f_{\infty}'(y) - \frac{\partial u}{\partial x} - f_{\infty}(y)\,\frac{\partial u}{\partial y}\right)\,dx\wedge dy > 0.
\]
\noindent Restricting to the circle $C = \{(x,0)\mid x \in \mathbb{S}^1\}$ where $f_{\infty}(0) = 0$ and $f_{\infty}'(0) = 0$, the inequality becomes $\partial u/\partial x < 0$ along $C$, which is impossible by periodicity.\\

\noindent Now let $f \in \mathcal{C}_{\varepsilon}(f_{\infty})$ with $0 < \varepsilon < \tfrac{1}{4}$. The characteristic foliation $(S_f)_\xi$ is nonsingular and has exactly two periodic orbits: $C_1 = \mathbb{S}^1 \times \{0\}$ and $C_2 = \mathbb{S}^1 \times \{-\tfrac{1}{2}\}$. Since $f'(0) \neq 0$ and $f'(-\tfrac{1}{2}) \neq 0$, both orbits have non-degenerate Poincaré return maps. Hence the foliation is almost Morse–Smale (since the characteristic foliation has no critical points, being Morse–Smale is equivalent to requiring all closed orbits to be non-degenerate). By a theorem of Giroux \cite{Gi91}, it follows that the surface $S_f$ is convex.
\end{proof}

\begin{thm}\label{Thm_Convexity_Flexibility}
    For all $0<\varepsilon<\frac{1}{4}$ and $f\in\mathcal{C}_{\varepsilon}(f_{\infty})$, there exists a contact homeomorphism $\Phi:Y\rightarrow Y$ that maps a convex surface $S_{f}$ to a non-convex surface $S_{f_{\infty}}$.
\end{thm}

\subsection{Proof of Theorem \ref{Thm_Convexity_Flexibility}}\label{subsec.pf_of_flexibility}

\noindent Let $Y=\mathcal{U}_1\supset \mathcal{U}_2\supset\ldots\supset\mathrm{Im}\,S_f$ be a decreasing sequence of open sets such that
\[\bigcap_{i\geq 1}\mathcal{U}_i=\mathrm{Im}\,S_f.\]
\noindent We will inductively construct a sequence of contactomorphisms $\{\Psi_i\}_{i=1}^{\infty}$ such that:
\begin{enumerate}[label=$(\mathcal{I}\arabic*)$]
    \item $\mathrm{supp}\,\Psi_i\subset\Phi_{i-1}(\mathcal{U}_i)$, $\Phi_0=\mathrm{Id}$ and $\Phi_i=\Psi_{i}\circ\cdots\circ\Psi_1$,
    \item $\Phi_i\circ S_f=S_{f_i}$, for a sequence $\{f_i\}_{i=1}^{\infty}\in\mathcal{C}_{\varepsilon}(f_{\infty})$ with $f_i\xrightarrow[i\to\infty]{C^0} f_{\infty}$,
    \item $d_{C^0}(\Psi_i,\mathrm{Id})<c/2^i$ for some universal constant $c>0$.
\end{enumerate}

\noindent We now explain how to conclude the proof of the Theorem \ref{Thm_Convexity_Flexibility}, provided we completed the inductive construction. From $(\mathcal{I}3)$, the sequence $\{\Phi_i\}_{i=0}^\infty$ is Cauchy in the $C^0$-topology and thus $C^0$-converges to a continuous map $\Phi: Y \to Y$. By $(\mathcal{I}2)$, $\Phi$ satisfies $\Phi \circ S_f = S_{f_\infty}$. To prove $\Phi$ is a homeomorphism, we show injectivity. Suppose $\Phi(p) = \Phi(q)$ for $p,q \in Y$:

\begin{itemize}
    \item \underline{$p,q \in \mathrm{Im}\,S_f$:} Since $\Phi|_{\mathrm{Im}\,S_f}$ is injective by $\Phi \circ S_f = S_{f_\infty}$, we have $p = q$.
    
    \item \underline{$p,q \in Y \setminus \mathrm{Im}\,S_f$:} For large $i$, $p,q \notin \mathcal{U}_i$, so by $(\mathcal{I}3)$, $\Phi_i(p) = \Phi(p)$ and $\Phi_i(q) = \Phi(q)$. As $\Phi_i$ is a diffeomorphism, $p = q$.
    
    \item \underline{$p \in \mathrm{Im}\,S_f$, $q \notin \mathrm{Im}\,S_f$:} For large $i$, $q \notin \overline{\mathcal{U}_i}$ and $\Phi_i(q) = \Phi(q) \notin \Phi_i(\mathcal{U}_i)$. However, $(\mathcal{I}1)$ implies $\Phi_j(p) \in \Phi_i(\overline{\mathcal{U}_i})$ for $j \geq i$, so $\Phi(p) \in \Phi_i(\overline{\mathcal{U}_i})$, a contradiction.
\end{itemize}

\noindent Thus, $\Phi$ is a homeomorphism that is a $C^0$-limit of a sequence of contactomorphisms, therefore $\Phi$ is a contact homeomorphism which moreover satisfies $\Phi\circ S_f=S_{f_{\infty}}$.\\

\noindent Let us show that we can construct a sequence with properties $(\mathcal{I}1-3)$ which satisfies additional technical conditions:

\begin{enumerate}[label=(\arabic*)]
    \item $f_i\in C^{\infty}_{\varepsilon_i}(f_{\infty})$ for a sequence $\{\varepsilon_i>0\}$ with $\varepsilon_i<C/4^{i}$,
    \item $|f_i(x)|>|f_{\infty}(x)|$ for $0<|x|<\varepsilon_i$,
    \item $\{(x,y,r)\mid\min\{f_{\infty}(y),f_i(y)\}\leq r\leq\max\{f_{\infty}(y),f_i(y)\}\}\subset\Phi_{i-1}(\mathcal{U}_i)$.
\end{enumerate}

\noindent Assume we have constructed the sequences up to the index $i$. Consider the set
\[\mathcal{U}:=\Phi_{i}(\mathcal{U}_{i+2})\supset\mathrm{Im}\, S_{f_i}.\]

\noindent For $f,g\in C^{\infty}(\mathbb{S}^1)$, define the interpolating set between $S_f$ and $S_g$:
\begin{equation}\label{Sigma_Surface}
\begin{aligned}
\Sigma_{\varepsilon}(f,g)
&= \bigcup_{\lambda \in [0,1]} S_{(1-\lambda)f + \lambda g} \cap \{|y| < \varepsilon\} \\
&= \{(x,y,r) \mid x \in \mathbb{S}^1,\, |y| < \varepsilon,\, 
   \min\{f(y),g(y)\} \leq r \leq \max\{f(y),g(y)\} \}.
\end{aligned}
\end{equation}

\begin{claim}
There exists a contactomorphism 
$\Psi'_{i+1}$, supported in $Op\big(\Sigma_{\varepsilon_i}(f_i,f_{\infty})\big)$, 
that fixes $S_{f_i}$ pointwise, has $C^0$–norm at most $C/2^i$, and satisfies
\[
\Psi'_{i+1}(\mathcal{U}) \supset \Sigma_{\varepsilon_i}(f_i,f_{\infty}).
\]
\end{claim}

\begin{proof}
\noindent Let us fix some positive number $0<\varepsilon_i'< \varepsilon_i$. Let $g \in \widetilde{\mathcal{C}}_{\varepsilon_i'}(f_{\infty})$ be such that 
\[
\Sigma_{\varepsilon_i}(f_i,g) \subset \mathcal{U}
\]
(this can be arranged by taking $g$ $C^0$–close to $f_i$) and 
\[
\forall\, \delta<|x|<\varepsilon_i' \quad |f_i(x)| > |g(x)| > |f_{\infty}(x)|.
\]

\noindent By Proposition~\ref{MainProposition}, applied to 
$f_{\infty}, g \in \widetilde{\mathcal{C}}_{\varepsilon_i}(f_{\infty})$, there exists a contactomorphism $\Phi^1$, the time–$1$ map of a $50\sqrt{\varepsilon_i}$–PSI isotopy $\{\Phi^\tau\}$ supported in $Op\big(\Sigma_{\varepsilon_i}(f_{\infty},g)\big)$, such that $\Phi^1(S_g) = S_{f_{\infty}}$ and $d_{C^0}(\Phi^1,\mathrm{Id}) < 100\sqrt{\varepsilon_i}$. Moreover, we may assume $\Phi^{\tau}|_{S_{f_i}} = \mathrm{Id}$, since $\Phi^{\tau}$ is generated by a Hamiltonian $H_{\tau} = \alpha(\partial_{\tau}\Phi^{\tau})$ that can be cut off so that its flow fixes $S_{f_i}$. Finally, set $\Psi'_{i+1} := \Phi^1$.
\end{proof}

\noindent Applying the previous claim yields $\Psi'_{i+1}$ with the desired properties. Let $\varepsilon_{i+1} < C/4^{i+1}$, and choose
\[f_{i+1} \in \mathcal{C}_{\varepsilon_{i+1}}(f_{\infty}) \subset \mathcal{C}_{\varepsilon_i}(f_{\infty})\]
\noindent such that $|f_i(x)|>|f_{i+1}(x)| > |f_{\infty}(x)|$ for $\delta' < |x| < \varepsilon_{i+1}$. We then apply Proposition~\ref{MainProposition} to the pair 
$f_i, f_{i+1} \in \mathcal{C}_{\varepsilon_i}(f_{\infty})$ to get a contactomorphism $\Psi''_{i+1}$ with the properties listed in the Proposition~\ref{MainProposition}. Finally, define
\[\Psi_{i+1}=\Psi_{i+1}''\circ\Psi_{i+1}'.\]
\noindent The triangle inequality implies that $\Psi_{i+1}$ satisfies $(\mathcal{I}3)$. The properties $(\mathcal{I}2)$, as well as (1) and (2), hold immediately by construction. Property $(\mathcal{I}1)$ follows from property (3). To see property (3), note that $\Psi'_{i+1}(\mathcal{U}) \supset \Sigma_{\varepsilon_i}(f_i,f_{\infty})$, and that $\Psi''_{i+1}$ is supported inside $Op(\Sigma_{\varepsilon_i}(f_i,f_{i+1}))$. This completes the induction step.

\begin{prop}\label{MainProposition}
Let $0 < \varepsilon < \tfrac{1}{4}$, and let $f_0, f_1 \in \mathcal{C}_{\varepsilon}(f_{\infty})$ or $f_0, f_1 \in \widetilde{\mathcal{C}}_{\varepsilon}(f_{\infty})$.  
Then there exists a contactomorphism $\Phi \in \mathrm{Cont}_c(Y, \ker\alpha)$ such that:
\begin{itemize}
    \item $\Phi(S_{f_0}) = S_{f_1}$,
    \item $\Phi$ is supported inside a neighbourhood $Op(\Sigma_{\varepsilon}(f_0, f_1))$,
    \item $d_{C^0}(\Phi, \mathrm{Id}) < 100\sqrt{\varepsilon}$, and $\Phi$ is the time-1 map of a contact $50\sqrt{\varepsilon}$–PSI isotopy $\{\Phi^{\tau}\}$.
\end{itemize}
\end{prop}

\subsection{Proof of Proposition \ref{MainProposition}}

A smooth non-singular 1-dimensional foliation $\mathcal{F}$ on the torus $\mathbb{T}(x,y)$ is \textit{admissible} if none of its leaves has a tangent line parallel to the $y$-axis. Equivalently, each point $(x,y)\in\mathbb{T}(x,y)$ has a well-defined slope $\mathrm{slope}_{\mathcal{F}}(x,y)\in\mathbb{R}$, given by the tangent line to the leaf through $(x,y)$. Every admissible foliation $\mathcal{F}$ thus defines a surface
\begin{equation*}
    S_{\mathcal{F}}:=\{(x,y,\mathrm{slope}_{\mathcal{F}}(x,y))\mid (x,y)\in\mathbb{T}\}\subset Y.
\end{equation*}
\noindent Moreover, the projection of the \textit{characteristic foliation} $(S_{\mathcal{F}})_\xi$ onto $\mathbb{T}$ coincides with $\mathcal{F}$.\\

\noindent Consider two closed curves $C_1,C_2\subset\mathbb{T}$ defined by
\begin{equation*}
    C_1=\mathbb{S}^1\times\{0\},\quad C_2=\mathbb{S}^1\times\{-1/2\}.
\end{equation*}

\noindent Note that $C_1$ and $C_2$ divide $\mathbb{T}$ into annuli:
\begin{equation*}
    A^{+}=\{0<y<1/2\},\quad A^{-}=\{-1/2<y<0\}.
\end{equation*}

\noindent We focus on an $\varepsilon$-neighbourhoods of the leaf $C_1$ and define:
\begin{equation*}
    A^+(\varepsilon)=\mathbb{S}^1\times(0,\varepsilon],\quad A^-(\varepsilon)=\mathbb{S}^1\times[-\varepsilon,0).
\end{equation*}

\begin{figure}[h]
    \centering
    \includegraphics[scale=0.9]{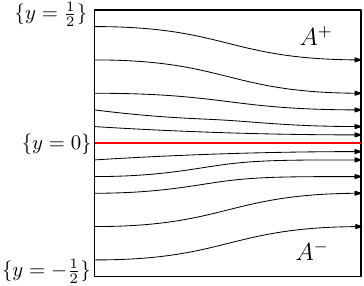}
    \caption{A foliation on the torus.}
    \label{fig:torus_foliation}
\end{figure}

\noindent For $i\in\{0,1\}$, consider the projection of the characteristic foliation $(S_{f_i})_{\xi}$ onto $A^{\pm}\subset\mathbb{T}$. The projection of each leaf intersects $\mathbb{S}^1\times\{\pm\varepsilon\}$ transversally and wraps around $A^{\pm}$, converging to $C_1$ forward in time and $C_2$ backward in time (see Figure \ref{fig:torus_foliation}). More precisely, for each $\theta\in\mathbb{S}^1$, the projection of the characteristic leaf of $(S_{f_i})_{\xi}$ passing through $(\theta,\pm\varepsilon)\in A^{\pm}$ admits a parametrization:
\begin{equation}\label{Leaf}
    \mathcal{L}^{\pm}_{f_i}:\mathbb{S}^1\times\mathbb{R}\rightarrow A^{\pm},\quad\mathcal{L}^{\pm}_{f_i}(\theta,t)=(\theta+t,y_{\pm}(t)),
\end{equation}

\noindent where $y_{\pm}(t)$ solves the ODE
\begin{equation}\label{ODE}
    y_{\pm}'(t)=f_i(y_{\pm}(t)),\quad y_{\pm}(0)=\pm\varepsilon.
\end{equation}

\noindent Note $\mathcal{L}_{f_i}^{\pm}$ is a diffeomorphism restricting to a diffeomorphism $\mathbb{S}^1\times[0,\infty)\xrightarrow{\cong} A^{+}(\varepsilon)$.\\

\noindent Let $T_0,T_1\subset\mathbb{S}^1=\mathbb{R}/\mathbb{Z}$ satisfy:
\begin{itemize}
    \item $T_0\cap T_1=\varnothing$; and $\mathbb{S}^1\setminus T_i,\,i\in\{0,1\}$ is a union of intervals shorter than $3\sqrt{\varepsilon}$;

    \item $\mathbb{S}^1\setminus (T_0\cup T_1)$ is a disjoint union of open intervals shorter than $\sqrt{\varepsilon}$.
\end{itemize}

\noindent Define open sets $\mathcal{U}_0,\mathcal{U}_1\subset\mathbb{T}$ by
\begin{equation*}
    \mathcal{U}_0=(\mathbb{S}^1\setminus T_0)\times Op([-\varepsilon,\varepsilon]),\quad 
    \mathcal{U}_1=(\mathbb{S}^1\setminus T_1)\times Op([-\varepsilon,\varepsilon]).
\end{equation*}

\noindent Let $\lambda:\mathbb{S}^1\times[0,1]\rightarrow [0,1]$ be smooth with:
\begin{itemize}
    \item $\lambda(\cdot,0)=0$ and $\lambda(\cdot,1)=1$,
    \item $\lambda(x,\tau)=0$ for $(x,\tau)\in T_0\times[0,1/2]$ and $\lambda(x,\tau)=1$ for $(x,\tau)\in T_1\times[1/2,1]$,
    \item $\big|\frac{\partial\lambda}{\partial x}(x,\tau)\big|<\frac{2}{\sqrt{\varepsilon}}$ for all $(x,\tau)\in\mathbb{S}^1\times[0,1]$.
\end{itemize}

\noindent Finally, define a $1$-parameter family $\phi^{\tau}:\mathbb{T}\rightarrow\mathbb{T}$ for $\tau\in[0,1]$ as (see Figure \ref{fig:mixed_foliation})
\begin{equation}\label{phi^t}
    \phi^{\tau}(q)=\begin{cases}
        q, & q\in C_1\cup C_2,\\
        (1-\lambda(x,\tau))\,q+\lambda(x,\tau)\,\mathcal{L}^{\pm}_{f_1}\circ(\mathcal{L}^{\pm}_{f_0})^{-1}(q), & q=(x,y)\in A^{\pm}.
    \end{cases}
\end{equation}

\begin{figure}[h]
    \centering
    \includegraphics[scale=1]{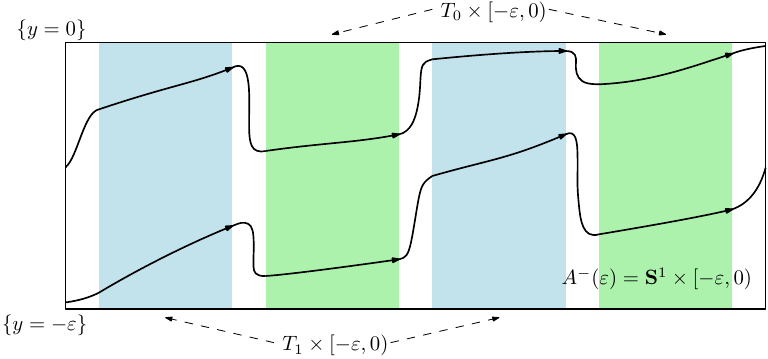}
    \caption{Characteristic foliation on $\phi^\tau(\mathbb{T})$ for $\tau=1/2$. A leaf alternates between characteristic foliation on $S_{f_0}$ and $S_{f_1}$.}
    \label{fig:mixed_foliation}
\end{figure}

\begin{claim}
    The map $\phi^{\tau}$ is a diffeomorphism for all $\tau\in[0,1]$.
\end{claim}

\begin{proof}
Since $\mathcal{L}^{\pm}_{f_1}\circ(\mathcal{L}^{\pm}_{f_0})^{-1}$ is a diffeomorphism of $A^{\pm}$ preserving the $x$-coordinate, it follows that $\phi^{\tau}|_{A^{\pm}}$ is a diffeomorphism. Also, $\phi^{\tau}(q) = q$ outside $A^{+}(\varepsilon)\cup A^{-}(\varepsilon)$, so $\phi^{\tau}|_{\mathbb{T} \setminus C_1}$ is a diffeomorphism onto its image. It remains to verify smoothness along $C_1$. By construction,
\[
\phi^{\tau}(\{\theta\}\times \mathbb{S}^1) = \{\theta\} \times \mathbb{S}^1,\quad \forall\theta\in\mathbb{S}^1.
\]
\noindent To describe $\phi^{\tau}$ near $C_1$, we use (\ref{ODE}) and the fact that $f_0(y)=f_1(y) = -y$ or $-y^3$ for $|y|\leq\delta$ (by (C3), ($\widetilde{C3}$)). We distinguish two cases:

\begin{enumerate}
    \item[\textbf{(1)}] If $f_0(y)=f_1(y) = -y$ on $|y| < \delta$, then for some constants $c_0, c_1 > 0$, the leaves are
    \[
    \mathcal{L}^{+}_{f_i}(\theta, t) = (\theta + t,\, c_i e^{-t}),\quad \mathcal{L}^{-}_{f_i}(\theta, t) = (\theta + t,\,-c_i e^{-t}).
    \]
    Substituting into (\ref{phi^t}), we obtain
    \[
    \phi^{\tau}(x, y) = \left(x,\, \left(1 + \frac{c_1 - c_0}{c_0}\lambda(x, \tau)\right)y\right).
    \]

    \item[\textbf{(2)}] If $f_0(y)=f_1(y) = -y^3$ on $|y| < \delta$, then for some constants $c_i\in\mathbb{R}$, the leaves are
    \[
    \mathcal{L}^{+}_{f_i}(\theta, t) = \left(\theta + t,\, \frac{1}{\sqrt{c_i + 2t}}\right),\quad
    \mathcal{L}^{-}_{f_i}(\theta, t) = \left(\theta + t,\,-\frac{1}{\sqrt{c_i + 2t}}\right).
    \]
    This yields
    \[
    \phi^{\tau}(x, y) = \left(x,\, (1 - \lambda(x, \tau))y + \lambda(x, \tau)\cdot \frac{y}{\sqrt{(c_1 - c_0)y^2 + 1}}\right).
    \]
\end{enumerate}

\noindent In both cases, the map is a smooth diffeomorphism near $y = 0$.
\end{proof}

\noindent Let $\mathcal{F}_0$ be the 1-dimensional foliation on $\mathbb{T}$ given as the $\mathbb{T}$-projection of the characteristic foliation on $S_{f_0}$. For $\tau\in[0,1]$, define a new foliation $\mathcal{F}_\tau$ by
\begin{equation}\label{F_tau}
    \mathcal{F}_{\tau} := \{\phi^{\tau}(L)\}_{L\in\mathcal{F}_0}.
\end{equation}

\noindent The foliations $\mathcal{F}_\tau$ and the maps $\phi^{\tau}:\mathbb{T} \to \mathbb{T}$ satisfy:
\begin{enumerate}
    \item $S_{\mathcal{F}_0} = S_{f_0}$ and $S_{\mathcal{F}_1} = S_{f_1}$;
    \item $\mathrm{supp}\,\phi^{\tau} \subset \mathcal{U}_0$ for $\tau \in [0,1/2]$, and $\mathrm{supp}\,\phi^{\tau} \subset \mathcal{U}_1$ for $\tau \in [1/2,1]$;
    \item the slope of any leaf in $\mathcal{F}_\tau|_{\mathbb{S}^1\times[-\varepsilon,\varepsilon]}$ is bounded by $3\sqrt{\varepsilon}$.
\end{enumerate}

\noindent The first two properties are immediate, so we verify the third one. Since $\mathcal{F}_\tau$ is obtained by applying $\phi^\tau$ to leaves of $\mathcal{F}_0$, it suffices to estimate
\[
\left|\,\mathrm{slope}\left\{t\mapsto \phi^\tau\left(\mathcal{L}^{\pm}_{f_0}(\theta,t)\right)\right\}\,\right| < 3\sqrt{\varepsilon}.
\]

\noindent From (\ref{phi^t}),
\[
\phi^\tau(\mathcal{L}^{\pm}_{f_0}(\theta,t)) = \left(\theta + t,\,(1 - \lambda(t+\theta,\tau))y_0(t) + \lambda(t+\theta,\tau)y_1(t)\right),
\]
where $y_0$, $y_1$ solve the ODE (\ref{ODE}). Therefore, the absolute value of the slope is
\begin{equation*}
\begin{aligned}
\Big|(1-\lambda(t+\theta,\tau))&\cdot y_0'(t)+\lambda(t+\theta,\tau)\cdot y_1'(t)+(y_1(t)-y_0(t))\frac{d}{dt}\lambda(t+\theta,\tau)\Big| \\
&\leq \max\{|y_0'(t)|,|y_1'(t)|\}+|y_1(t)-y_0(t)|\cdot\max_{(x,\tau)\in\mathbb{S}^1\times[0,1]}\Big\lvert\frac{\partial\lambda}{\partial x}(x,\tau)\Big\rvert  \quad \\
&\leq \max_{y\in[-\varepsilon,\varepsilon]}\max\{|f_0(y)|,|f_1(y)|\}+\varepsilon\cdot\frac{2}{\sqrt{\varepsilon}}<3\sqrt{\varepsilon}.
\end{aligned}
\end{equation*}

\noindent Finally, define the isotopy of embeddings $\varphi^{\tau}:\mathbb{T} \to (\mathbb{T} \times \mathbb{R},\xi)$ by
\begin{equation*}
    \varphi^\tau(q) := \left(\phi^\tau(q),\, \mathrm{slope}_{\mathcal{F}_\tau}(\phi^\tau(q))\right).
\end{equation*}

\noindent From (\ref{F_tau}), the $\mathbb{T}$-projection of the characteristic foliation on $\mathrm{Im}\,\varphi^\tau = S_{\mathcal{F}_\tau}$ is precisely $\mathcal{F}_\tau = \{\phi^\tau(L)\}_{L\in\mathcal{F}_0}$. In particular, each $\varphi^\tau$ pulls back the characteristic foliation to $\mathcal{F}_0$:
\[
(\varphi^\tau)^*\left((\mathrm{Im}\,\varphi^\tau)_\xi\right) = \mathcal{F}_0.
\]

\noindent Moreover, since $||\mathrm{slope}_{\mathcal{F}_\tau} - \mathrm{slope}_{\mathcal{F}_0}||_\infty < 3\sqrt{\varepsilon}$, we have
\begin{equation}\label{ConditionPSI}
    \mathrm{supp}_{\tau\in[0,1/2]}\varphi^\tau \subset \mathcal{U}_0 \times (-3\sqrt{\varepsilon}, 3\sqrt{\varepsilon}),\quad
    \mathrm{supp}_{\tau\in[1/2,1]}\varphi^\tau \subset \mathcal{U}_1 \times (-3\sqrt{\varepsilon}, 3\sqrt{\varepsilon}).
\end{equation}

\noindent Therefore, $\varphi^\tau$ is a $(50\sqrt{\varepsilon})$-PSI, and by Proposition~\ref{PropositionPSI}, we obtain the desired isotopy $\Phi^t$ and final embedding $\Phi = \Phi^1$.

\end{document}